\documentclass[12pt]{amsart}

\newtheorem{theorem}{Theorem}[section]

\newtheorem{corollary}[theorem]{Corollary}

\newtheorem{proposition}[theorem]{Proposition}

\theoremstyle{definition}

\newtheorem{definition}[theorem]{Definition}

\newtheorem{remark}[theorem]{Remark}

\theoremstyle{parrafo}

\begin{document}

\title[]{A characterization of the uniform strong type $(1,1)$ bounds for averaging operators}

\author{J. M. Aldaz}
\address{Instituto de Ciencias Matem\'aticas (CSIC-UAM-UC3M-UCM) and Departamento de 
Mate\-m\'aticas,
Universidad  Aut\'onoma de Madrid, Cantoblanco 28049, Madrid, Spain.}
\email{jesus.munarriz@uam.es}
\email{jesus.munarriz@icmat.es}


\thanks{The author was partially supported by Grant MTM2015-65792-P of the
MINECO of Spain, and also by by ICMAT Severo Ochoa project SEV-2015-0554 (MINECO)}







\begin{abstract} We prove that in a  metric measure space $(X, d, \mu)$, the averaging operators $A_{r, \mu }$ satisfy a  uniform strong type $(1,1)$ bound 
$
\sup_{r, \mu}\|A_{r, \mu }\|_{L^1\to L^{1}} < \infty$ if and only if  $X$ satisfies
a certain geometric condition, 
the equal radius Besicovitch  intersection property.  
\end{abstract}


\maketitle


\markboth{J. M. Aldaz}{Bounds for averaging operators}

\section {Introduction} 
Motivated by a question of Prof. Przemys\l{}aw G\'orka (personal communication) we show that averaging operators are of strong type (1,1) for arbitrary, locally finite $\tau$-additive Borel measures $\mu$ on a metric space $X$, with bounds   independent of 
 $\mu$ and of $r$, if and only if $X$ has a certain property of Besicovitch type, called here the equal radius Besicovitch  intersection property, cf. Definition \ref{BIP} for the precise statement.

 This characterization,   obtained via minor modifications of the arguments from \cite{Al1} and \cite{Al2}, is entirely analogous to the one presented in \cite{Al1}
for the centered maximal operator, which uses the Besicovitch intersection property, a stronger condition.
Thus, we conclude that uniform weak type $(1,1)$ bounds for the centered maximal operator are stronger than  uniform strong type $(1,1)$ bounds for the averaging operators. Since for Banach spaces the equal radius Besicovitch  intersection property is equivalent to the Besicovitch  intersection property, we obtain several sharp bounds on $\mathbb{R}^d$ for $
\sup_{r, \mu}\|A_{r, \mu }\|_{L^1\to L^{1}}$, by direct transference from the maximal function case. This allows us to improve previously known upper bounds for the standard gaussian measures in euclidean spaces,  cf. \cite{Al3}.

\section {Definitions and results} 

We will use $B^{o}(x,r) := \{y\in X: d(x,y) < r\}$ to denote metrically open balls, 
 and 
$B^{cl}(x,r) := \{y\in X: d(x,y) \le r\}$ to refer to metrically closed balls;
open and closed will
always be understood  in the metric (not the topological) sense. 
 If we do not want to specify whether balls are open or  closed,
we write $B(x,r)$. But when we utilize $B(x,r)$,  all balls are taken to be of the same kind, i.e., all open or all closed. Also, whenever we speak of balls,
we assume that suitable centers and radii have  been chosen. 

\begin{definition} Let $(X, d)$ be a metric space. A Borel measure $\mu$ 
	on $X$ is   {\em $\tau$-additive} or {\em $\tau$-smooth}, if for every
	collection  $\{O_\alpha : \alpha \in \Lambda\}$
	of  open sets, we have
	$$
	\mu (\cup_\alpha O_\alpha) = \sup_{\mathcal{F}} \mu(\cup_{i=1}^n O_{\alpha_i}),
	$$
	where the supremum is taken over all finite subcollections $\mathcal{F} = \{O_{\alpha_1}, \dots, O_{\alpha_n} \}$
	of  $\{O_\alpha : \alpha \in \Lambda\}$.
	If $\mu$  assigns finite measure
	to bounded Borel sets, we say it is {\em locally finite}.
	Finally, we call $(X, d, \mu)$  a {\em metric measure space} if
	$\mu$ is a  $\tau$-additive, locally finite  Borel measure on the metric space $(X, d)$. 
	\end{definition}

The preceding definition includes all locally finite Borel measures on 
separable metric spaces and all Radon measures on arbitrary metric spaces. From now on we always suppose that measures are locally finite, 
not identically zero, and  that metric spaces
have at least two points.

Recall that 
the complement of
the support $(\operatorname{supp}\mu)^c := \cup \{ B^{o}(x, r): x \in X, \mu B^{o}(x,r) = 0\}$
of a Borel  measure
$\mu$,  is an open set, and hence measurable. 

\begin{definition}\label{full} Let $(X, d)$ be a metric space and let
	$\mu$ be a locally finite Borel measure on $X$. 
	If $\mu (X \setminus \operatorname{supp}\mu) = 0$, 
	we say that $\mu$ has {\em full support}. 
\end{definition}

By $\tau$-additivity,  if $(X, d, \mu)$ is  a metric measure space, then 
$\mu$ has full support,
since $X \setminus \operatorname{supp}\mu $ is a union of open balls of measure zero.

\begin{definition}\label{maxfun} Let $(X, d, \mu)$ be a metric measure space and let $g$ be  a locally integrable function 
	on $X$. For each fixed $r > 0$ and each $x\in \operatorname{supp}\mu$, the
	averaging operator $A_{r, \mu}$ is defined as
	\begin{equation}\label{avop}
	A_{r , \mu} g(x) := \frac{1}{\mu
		(B(x, r))} \int _{B(x, r)}  g \ d\mu.
	\end{equation}
\end{definition}

Averaging operators in metric measure spaces are defined almost everywhere,   by
$\tau$-additivity. Sometimes it is convenient to  specify whether balls are open or closed; in that case,
we use $A_{r , \mu}^{o} $ and $A_{r , \mu}^{cl} $ for the corresponding operators. 
Furthermore, when we are considering only one measure $\mu$ we often omit it, 
writing $A_{r } $ instead of the longer $A_{r , \mu} $.

Recall that given $p$ with $1 \le p < \infty$,
$A_{r,\mu}$ satisfies a
weak type $(p,p)$ inequality if there exists a constant $c > 0$ such that
\begin{equation}\label{weaktypep}
\mu (\{A_{r,\mu} g \ge \alpha\}) \le \left(\frac{c \|g\|_{L^p(\mu)}}{\alpha}\right)^p,
\end{equation}
where $c=c(p,  \mu)$ depends neither on $g\in L^p (\mu)$
nor on $\alpha > 0$. The lowest constant $c$ that satisfies the preceding
inequality is denoted by $\|A_{r,\mu}\|_{L^p\to L^{p, \infty}}$.
Likewise, if there exists a constant $c > 0$ such that
\begin{equation}\label{strongtypep}
\|A_{r,\mu} g \|_{L^p(\mu)}  \le c \|g\|_{L^p(\mu)},
\end{equation}
we say that $A_{r,\mu}$ satisfies a
strong type $(p,p)$ inequality.  The lowest such constant (the operator norm) 
is denoted by $\|A_{r,\mu}\|_{L^p\to L^{p}}$.

\begin{definition} \label{conjugate} We call 
	\begin{equation}\label{conjug}
	a_r (y)  
	: =
	\int_{\operatorname{supp(\mu)}}  \frac{\mathbf{1}_{B(y,r)}(x)}{\mu B(x,r)}  \  d\mu(x) 
	\end{equation}
	the {\em conjugate function} to the averaging operator $A_{r}$.
\end{definition}

Note that the conjugate function $a_r$ is well defined a.e., whenever $y$ belongs to
the support of $\mu$. According to \cite[Theorem 3.3]{Al2}, 
$A_r$ is bounded on
$L^1(\mu)$ if and only if $a_r \in L^\infty (\mu)$, in which case 
$\|A_r \|_{L^1(\mu)\to L^1(\mu)} = \|a_r\|_{L^\infty(\mu)\to L^\infty (\mu)} $. 
We will  use $a_{r }^{o} $ and $a_{r }^{cl} $ to  specify whether balls are open or closed.

\begin{definition}  \label{BIP} A  collection $\mathcal{C}$ of balls in a metric space $(X, d)$ is a {\em Besicovitch family} if 
for every pair  of distinct balls $B(x, r), B(y,s) \in \mathcal{C}$,
$x \notin B (y,s)$ and  $y \notin B (x,r)$. Denote by
$\mathcal{EBF}(X,d)$ the collection of all Besicovitch families $\mathcal{C}$
of $(X, d)$ with the property that all balls in
$\mathcal{C}$ have equal radius.  The 
{\em equal radius Besicovitch constant}  of $(X, d)$ is
\begin{equation}\label{BC}
E (X, d) := \sup \left\{	\sum_{B(x,  r) \in \mathcal{C}} \mathbf{1}_{B(x,  r)} (y): y \in X, \ \mathcal{C} \in \mathcal{EBF}(X,d) \right\}.
\end{equation}
We say that $(X, d)$ has the  {\em equal radius Besicovitch Intersection Property} with constant $E (X, d)$  if $E (X, d) < \infty$.
The 
{\em Besicovitch constant}  $L(X, d)$ is defined in the same way, save that the restriction that all balls in each collection have the same radius is lifted. 
We say that $(X, d)$ has the  {\em Besicovitch Intersection Property}   if $L (X, d) < \infty$.
\end{definition}

\begin{definition} \label{geomdoub} A metric space is {\it geometrically doubling}  if there exists a positive
integer $D$ such that every ball of radius $r$ can be covered with no more than $D$ balls
of radius $r/2$.  We call the smallest such $D$ the {\em doubling constant} of the space.
\end{definition}

\begin{remark} Call a  Besicovitch family  $\mathcal{C}$ {\em intersecting} if 
	$\cap \mathcal{C} \ne \emptyset$. It is well known that if $X$ is geometrically doubling with constant $D$, then $D$ is an upper bound for the cardinality of any intersecting Besicovitch family $\mathcal{C}$ with equal radius  $r$. To see why, consider any $y\in \cap \mathcal{C}$, and note that the centers of all balls in $\mathcal{C}$ form an $r$-net in $B(y,r)$;
we use the convention that $r$-nets are strict when dealing with closed balls, so the distance between any two points in the net is striclty larger than $r$, and non-strict when dealing with open balls. Cover $B(y,r)$ with  $\le D$ balls of radius $r/2$. Since each such ball contains at most the center of one ball from $\mathcal{C}$, the result follows. Thus, we always have $E (X, d) \le D$.

Geometrically doubling does not, by itself, imply the Besicovitch intersection property: a well known example
is given by  the Heisenberg groups
$\mathbb{H}^n$ with the
Kor\'any metric: cf.  \cite[pages 17-18]{KoRe} or
\cite[Lemma 4.4]{SaWh}. Thus, the Heisenberg groups provide a natural example of spaces where the equal radius Besicovitch intersection property holds and the Besicovitch intersection property fails.
\end{remark}

 The next proposition, for collections without the equal
radius restriction, appears in \cite[Proposition 2.4]{Al1}. 

\begin{proposition}  A metric space $(X, d)$ has the equal radius  Besicovitch
		intersection property with constant $E$  for collections of open balls, if and only if it
		 has the  Besicovitch
		 intersection property for collections of closed  balls, with the same constant.
		\end{proposition}
		
\begin{proof} Denote by $E^o$ and $E^c$ the lowest constants
for collections of open balls and for collections of closed balls, respectively. 
Suppose first that  $E^o < \infty$. Let $\mathcal{C}$ be an intersecting  
 Besicovitch
family   of closed balls, all of which have the same radius $r$, and  select any finite subcollection $\{B^{cl} (x_1, r), \dots , B^{cl} (x_N, r)\}$. It is enough to
prove that $N \le E^o$. Let 
$t := \min\{d(x_j, x_i) : 1 \le i < j \le N\}$. Since $t > r$, it follows that 
$\{B^o (x_1, t), \dots , B^o (x_N, t)\}$ is an intersecting equal radius Besicovitch family
of open balls,
so $N \le E^o$.

Suppose next that  $E^c < \infty$. Let $\mathcal{C}$ be an intersecting
Besicovitch
family   of open balls, with equal radius  $r$. Select $y \in \cap \mathcal{C}$, and then choose $\delta > 0$ so small that 
for every ball $B^{o}(x,r) \in \mathcal{C}$,
we have  $y \in B^{o}(x, r- \delta)$. 
 Then
the collection $\{B^{cl}(x,r - \delta): B^{o}(x,r) \in \mathcal{C}\}$ is an intersecting equal radius Besicovitch family 
of closed balls, so its cardinality is bounded by $E^c$.
	\end{proof}

It is shown in \cite[Theorem 2.5]{Al1} that the existence of uniform weak type $(1,1)$ bounds for the centered maximal operator is equivalent to the Besicovitch intersection property. More precisely

\begin{theorem}\label{Borelequiv} Let $(X, d)$ be a  
	metric  space.  The following are equivalent:
	
	1)  $(X, d)$ has the  Besicovitch intersection property with constant $L$.
	
	2) For every $\tau$-additive, locally finite  Borel measure $\mu$ on $X$, the centered maximal operator
	associated to $\mu$ satisfies 
	$\|M_{\mu}\|_{L^1  \to L^{1,\infty}} \le 
	L$.	
\end{theorem}

The situation regarding the existence of  strong type $(1,1)$ bounds for the averaging operators 
$A_{r, \mu}$, uniform in both $r$ and $\mu$ is entirely analogous, but with the equal radius Besicovitch intersection property replacing the  Besicovitch intersection property.

\begin{theorem}\label{discreteequiv} Let $(X, d)$ be a  
	metric  space.  The following are equivalent:
	
1)  The space $(X, d)$ has the  equal radius Besicovitch intersection property with constant $E$.

2) For every $r > 0$ and every $\tau$-additive, locally finite  Borel measure $\mu$ on $X$, we have
$\|A_{r, \mu}\|_{L^1  \to L^1} \le 
E$.	
	
3)  For every  $r > 0$ and every finite  weighted sum of Dirac deltas $\mu := \sum_{i = 1}^N c_i \delta_{x_i}$,
the averaging  operator  satisfies 
$\|A_{r, \mu}\|_{L^1  \to L^1}  \le 
E$.
\end{theorem}

\begin{proof} Let us  show that  1) $\implies$ 2). Disregarding a set of measure zero if needed, we suppose that 
$X = \operatorname{supp} \mu$, so every ball has positive measure.
Fix $y \in X$ and $r > 0$. First we consider  the open balls case. Let $0 < s < r$, let 
$$g(x) := \frac{\mathbf{1}_{B^o(y,r)}(x)}{\mu B^o(x,r)},   
\mbox{ \ and let \ }
g_s(x) := \frac{\mathbf{1}_{B^o(y,s)}(x)}{\mu B^o(x,r)}.
$$
 Since balls are open,
$g_s\uparrow g$ everywhere as $s\uparrow r$,
so we can use the monotone convergence theorem. Thus, it is enough to show that $\lim_{s\to r}
 \int_{X} g_s d\mu \le E(X,d)$ to conclude that
 $\|a^o_r\|_{L^\infty(\mu)\to L^\infty (\mu)} \le E(X,d)$. Then  the result follows, 
 since $\|A_r \|_{L^1(\mu)\to L^1(\mu)} = \|a_r\|_{L^\infty(\mu)\to L^\infty (\mu)} $ by  \cite[Theorem 3.3]{Al2}.	

Next we argue as in the proof of 	
 \cite[Theorem 3.5]{Al2}, which dealt with the case where $X$ is geometrically doubling. 
 	Note first  that $b_1 := \inf \{\mu B^o(x,r) : x \in B^o(y,s)\} > 0$. To see why, observe that for every $x\in B^o(y,s)$ and 
 every $w\in B^o(y, r - s)$, $d(x, w) \le d(x, y) + d (y, w) < s + r - s$,
so $ B^o(y, r - s)\subset B^o(x, r)$ and thus $0 < \mu  B^o(y, r - s) \le b_1$.
	Now take $0 < \varepsilon \ll 1$, and choose $u_1\in B^o(y,s)$ so that
	$\mu B^o(u_1, r) < (1 + \varepsilon) b_1$; let 
	$b_2 := \inf \{\mu B^o(x,r) : x \in B^o(y,s) \setminus B^o(u_1, r) \}$, and select
	$u_2\in B^o(y,s) \setminus B^o(u_1, r) $ so that $\mu B^o(u_2, r) < (1 + \varepsilon) b_2$;
	repeat, with 
	$b_{k + 1} := \inf \{\mu B^o(x,r) : x \in B^o(y,s) \setminus \cup_1^k B^o(u_i, r)  \}$, and
	$u_{k + 1}\in B^o(y,s) \setminus \cup_1^k B^o(u_i, r) $ so that
	$\mu B^o(u_{k + 1}, r) < (1 + \varepsilon) b_{k + 1}$. Since the 
	balls $B^o(u_i, r)$ form a Besicovitch family and all contain $y$, 
	there is an $m\le E(X,d)$ such that
	$B^o(y,s) \setminus \cup_1^m B^o(u_i, r) = \emptyset$, and then the process stops.

	Fix $x\in B^o(y,s)$, and let $i$ be the first index such that $x\in B^o(u_i, r)$. Then
	$$
	\frac{\mathbf{1}_{B^o(y,s)}(x)}{\mu B^o(x,r)} 
	\le 
	(1 + \varepsilon) \frac{\mathbf{1}_{B^o(y,s) \cap B^o(u_i ,r)}(x)}{\mu B^o(u_i ,r)} 
	\le 
	(1 + \varepsilon)  \sum_{j=1}^m\frac{\mathbf{1}_{B^o(y,s) \cap B^o(u_j ,r)}(x)}{\mu B^o(u_j ,r)},
	$$
	so
	\begin{equation*}
\int_X g_s d\mu   
	=
	\int_X  \frac{\mathbf{1}_{B^o(y,s)}(x)}{\mu B^o(x,r)}  \  d\mu(x) 
	\le
	\int_X  (1 + \varepsilon)  \sum_{j=1}^m\frac{\mathbf{1}_{B^o(y,s) \cap B^o(u_j ,r)}(x)}{\mu B^o(u_j ,r)}  \  d\mu(x) 
	\end{equation*}
	\begin{equation*}
	\le 
	(1 + \varepsilon)   \int_X \sum_{j=1}^m\frac{\mathbf{1}_{B^o(u_j ,r)}(x)}{\mu B^o(u_j , r)}  \  d\mu(x) 
	\le
	(1 + \varepsilon)  E (X, d),
	\end{equation*}
	and now
	$a^o_r (y) \le E (X,d) 
	$ follows by letting $\varepsilon \downarrow 0$ and  $s \uparrow r$.
	
The closed balls case is proven using the result for open balls.  Let $0 \le f \in L^1(\mu)$, 
$w\in X$, 
$R \ge 1$,  and  $\varepsilon > 0$. By monotone convergence, taking $R \uparrow \infty$, it is
enough to show that 
$$
\|\mathbf{1}_{B (w ,R)} \mathbf{1}_{\{ A^{cl}_{r} f \le R\}} A^{cl}_{r} f\|_{L^1} 
\le
\varepsilon + (1 + \varepsilon) E (X,d) \|f\|_{L^1}.
$$
For each $x\in B (w ,R)$, choose $r_x > 0$ so that 
$\mu B^{o}(x, r + r_x) < (1 + \varepsilon) \mu B^{cl}(x, r)$,
and  let
$E_n := \{x \in B(w, R) : r_x > 1/n\}$. Then select $N \gg 1$
satisfying  
$\mu \left(B(w, R)  \setminus E_N\right) < \varepsilon / R$.
Now for all $x\in E_N$, we have
$$
A^{cl}_{r} f (x)
\le
\frac{1}{ \mu B^{cl}(x, r)} \int_{B^{o}(x, r + 1/N) } f d\mu
$$
$$
\le
\frac{1 + \varepsilon}{ \mu B^{o}(x, r + 1/N) } \int_{B^{o}(x, r + 1/N)} f d\mu = (1 + \varepsilon) A^o_{r + 1/N} f(x),
$$
so
$$
\|\mathbf{1}_{B (w ,R)} \mathbf{1}_{\{ A^{cl}_{r} f \le R\}} A^{cl}_{r} f\|_{L^1} 
\le
\|\mathbf{1}_{B (w ,R) \setminus E_N} \mathbf{1}_{\{ A^{cl}_{r} f \le R\}} A^{cl}_{r} f\|_{L^1} + \|\mathbf{1}_{E_N} A^{cl}_{r} f\|_{L^1} 
$$
$$
\le \varepsilon + 
(1 + \varepsilon)\|A^o_{r + 1/N} f\|_{L^1} 
\le 
\varepsilon + (1 + \varepsilon) E (X,d) \|f\|_{L^1}.
$$

Since 3) is a special  case of 2), the only  implication left is 3) $\implies$ 1); we prove that
	if  $\mathcal{C}$  is an intersecting   Besicovitch family in $(X,d)$ of equal radius $r$ and
	cardinality $ > E$, then there exists a discrete measure $\mu_c$ with 
	finite support, for which 
	$\|A_{r, \mu_c}\|_{L^1\to L^{1}} >
	E$. 
	We may suppose
	that $\mathcal{C} = \{B(x_1, r), \dots ,B(x_{E + 1}, r)\}$
	by throwing away some balls if needed.
	Let $y \in \cap \mathcal{C}$, and for $0 < c \ll 1$, define 
	$\mu_c := c \delta_y + \sum_{i= 1}^{L + 1} \delta_{x_i}$. Set
	$f_c = c^{-1} \mathbf{1}_{\{y\}}$. Then $\|f_c\|_1 = 1$, and
	for $1 \le i \le E+ 1$, we have $A_{r, \mu_c} f_c (x_i) = 1/(1 + c)$,  while $A_{r, \mu_c} f_c (y) > 0$. Thus
	$$
	\|A_{r, \mu_c} f_c\|_1 >
\frac{E + 1}{ 1 + c},
	$$
	and the result follows by taking 
	$c$ small enough.
	\end{proof}

Since $\|A_{r,\mu} \|_{L^\infty(\mu)\to L^\infty(\mu)}  =1$, by  interpolation
or  by Jensen's inequality (cf. \cite[Theorem 2.10]{Al3})
for all  $1 < p < \infty$, we have $\|A_{r,\mu}  \|_{L^p(\mu)\to L^p(\mu)}  \le  E(X,d)^{1/p}$. 

\begin{remark} In addition to  having $\sup_{r, \mu}\|A_{r,\mu}  \|_{L^1(\mu)\to L^1(\mu)} =  E(X,d)$, using the same measures and functions it is easy to see that equality also holds for the weak type $(1,1)$ bounds, that is, 
$\sup_{r, \mu}\|A_{r,\mu}  \|_{L^1(\mu)\to L^{1, \infty}(\mu)} =  E(X,d)$. In fact, since the function $f_c$ is a scalar multiple of an indicator function, this equality holds in the restricted weak type (1,1) case.
\end{remark}

The preceding theorem entails that the uniform weak type $(1,1)$ of the centered
maximal operator is stronger than the uniform strong type $(1,1)$ of the averaging
operators.

\begin{corollary}  \label{weakstrong}
	Given any metric space   $(X, d)$, we have 
	$$
	\sup_{r, \mu}\|A_{r, \mu }\|_{L^1\to L^{1}} \le \sup_{\mu} \|M_{\mu}\|_{L^1  \to L^{1,\infty}},
	$$
	where the supremum on the left hand side is taken over all $r > 0$ and all $\tau$-additive, locally finite Borel measures $\mu$ on $X$, and the supremum on the right, over all such $\mu$.
\end{corollary} 

\begin{corollary}  \label{L1conv}
	If  $(X, d)$ has the equal radius Besicovitch intersection property, and $\mu$ is a
	$\tau$-additive Borel measure on $X$, then for  every $f\in L^p(\mu)$,  $1 \le p < \infty$,
	we have $\lim_{r\to 0}  A_{r} f  =  f$ in $L^p$.
Additionally, if  $(X, d)$ has the Besicovitch intersection property,  then 
$\lim_{r\to 0}  A_{r} f  =  f$ almost everywhere.
\end{corollary} 

The $L^p$ convergence follows in a standard fashion from the uniform boundedness of
the averaging operators (cf. \cite{Al2} for more details), while the a. e. convergence is a consequence of the weak type of the centered maximal operator. For homogeneous distances on homogeneous groups, the almost everywhere convergence had already appeared in \cite[Theorem 1.5]{LeRi}.

\vskip .2 cm

Analogously to the case of the centered maximal operator  (see \cite{Al1}) given any $p \in (1, \infty)$, the uniform weak
type $(p,p)$  implies the  equal radius Besicovitch intersection property, and consequently, one can extrapolate from uniform
weak type $(p,p)$ to uniform strong type $(1,1)$. 
Recall that the floor function $\lfloor x \rfloor$ denotes the integer part of $x$.

\begin{theorem}\label{extrapolation} Let $(X, d)$ be a  
	metric  space. Each of the following statements implies the next:
	
1)	There exist a $p$ with $1 < p < \infty$ and an integer $N\ge 1$, such that
for every discrete, finite  Borel measure $\mu$ with finite support
in $X$, and every $r > 0$, the averaging operators
$A_{r, \mu}$ satisfy 
	$\|A_{r, \mu}\|_{L^p\to L^{p,\infty}} \le N$.
	
2) The space $(X, d)$ has the  equal radius Besicovitch intersection property with constant 
$\lfloor p^p (p - 1)^{(1-p)} N^p \rfloor$.
	
	3)  For every $\tau$-additive, locally finite  Borel measure $\mu$ on $X$ and every $r > 0$,  the averaging operators
	$A_{r, \mu}$ satisfy 
	$\|A_{r, \mu}\|_{L^1\to L^{1}} \le 
	\lfloor p^p (p - 1)^{(1-p)} N^p \rfloor$.	
	\end{theorem}

\begin{proof} The implication   2) $\implies$ 3) is part of the preceding result.  Regarding
1) $\implies$ 2),  we show that
	if  $\mathcal{C}$  is an intersecting  Besicovitch family  in $(X,d)$, of
	cardinality  strictly larger than $\lfloor p^p (p - 1)^{(1-p)} N^p\rfloor$ and equal radius $r$, then there exists a finite
	sum of weighted Dirac deltas $\mu$, for which 
	$\|A_{r, \mu}\|_{L^p\to L^{p,\infty}} >
	N$. 
	
	Let $q = p/(p - 1)$ be the dual exponent of $p$, and let $J:= \lfloor p^p (p - 1)^{(1-p)} N^p \rfloor + 1$.	
	 We may suppose
	that $\mathcal{C} = \{B(x_1, r), \dots, B(x_{J}, r)\}$.
	Let $y \in \cap \mathcal{C}$, and set, for $ c > 0$, 
	$\mu_c :=  c \delta_y + \sum_{i= 1}^{J} \delta_{x_i}$.
	Recall that for every $\alpha > 0$, by definition of the weak $(p,p)$ constant $\|A_{r, \mu_c}\|_{L^p\to L^{p,\infty}}$,
	\begin{equation}\label{weaktypepM}
	\mu_c (\{A_{r, \mu_c} f \ge \alpha\}) \le \left(\frac{\|A_{r, \mu_c}\|_{L^p\to L^{p,\infty}} \|f\|_{L^p}}{\alpha}\right)^p.
	\end{equation}
	 Set
	$f = \mathbf{1}_{\{y\}}$; then $\|f\|_{L^p(\mu_c)} = c^{1/p}$. For $1 \le i \le J$, we have $A_{r, \mu_c} f (x_i) = c/(1 + c)$. 
	Thus,
	$\mu_c \{A_{r, \mu_c} f \ge  c/(1 + c)\} = J $,  so 
	with $\alpha = c/(1 + c)$, we have 
	\begin{equation}\label{weaktype}
	J^{1/p} 
	\le
	 \frac{\|A_{r, \mu_c}\|_{L^p\to L^{p,\infty}} (1 + c)}{c^{1/q}}.
	 	\end{equation}
	 Maximizing $g(c) = c^{1/q}/(1 + c)$ we get $c = p - 1$
	 and $g(p - 1) = (p - 1)^{(p - 1)/p} p^{-1}$,  so
	 \begin{equation}\label{weaktype1}
	N 
	=
	 \frac{(p -1)^{(p - 1)/p}  \left(p^p (p - 1)^{(1-p)} N^p\right)^{1/p} }{p} 
	< \frac{(p -1)^{(p - 1)/p}  \ J ^{1/p} }{p} 
	\le
	 \|A_{r, \mu_c}\|_{L^p\to  L^{p,\infty}}.
	 	\end{equation}
\end{proof}

\section{Consequences for $\mathbb{R}^d$}

In this section  we take balls to be closed. Unlike the case of the Heisenberg groups, where we have $E(X,d) < \infty$ and $L(X,d) = \infty$, in Banach spaces we always have $E(X,d)  = L(X,d)$.

\begin{theorem} \label{kiss}  If $(X, \|\cdot\|)$ is a
	Banach space, then  $E(X, \| \cdot\|) = L(X, \| \cdot\|)$. 
\end{theorem}

\begin{proof}  It suffices to show that 
$E(X, \| \cdot\|) \ge L(X, \| \cdot\|)$. Both $E (X, \| \cdot\|)$ and $L (X, \| \cdot\|)$ are defined as suprema, so it is enough to prove that given any finite, intersecting Besicovitch family 
$\mathcal{C} := \{B^{cl} (x_1, r_1),  \dots,  B^{cl} (x_n, r_n) \}$, we can produce an equal radius  intersecting 
Besicovitch family  of the same cardinality. Choose $y \in \cap \mathcal{C}$, 
	 and let $r_y := \min\{ \|x_1 - y\| ,  \dots,  \|x_n - y\|\}$.
By a translation and a dilation, we may assume that
$y = 0$ and $r_y = 1$. We claim that 
$\mathcal{C}^\prime := \{B^{cl} (x_1/\|x_1\|, 1),  \dots,  B^{cl} (x_n/\|x_n\|, 1) \}$ is a Besicovitch family. To show that any two vectors in
$ \{ x_1/\|x_1\|,  \dots,  x_n/\|x_n\|\}$  are at distance
$ > 1$,
we choose a pair of centers  $x_i$ and $x_j$ of balls from $\mathcal{C}$, with, say,  $\|x_i\| \ge  \|x_j\|$. 
Since $\|x_i - x_j\|  > \|x_i\|$, using the lower bound for the angular distances from   \cite[Corollary 1.2]{Ma}, we get
\begin{equation*}
\left\|\frac{x_i}{\|x_i\|}-\frac{x_j}{\|x_j\|}\right\|
\ge
\frac{\|x_i - x_j\|  - \left| \|x_i\| - \|x_j\|\right| }{\min\left\{\|x_i\|, \|x_j\|\right\} }
= \frac{\|x_i - x_j\|  - \|x_i\| + \|x_j\|}{ \|x_j\| } 
>
1.
\end{equation*}
 \end{proof}

As is the case with the maximal operator, cf. \cite[Theorem 3.3]{Al1}, in $\mathbb{R}^d$  it is possible to
construct a measure $\mu$ for which 
the supremum  is attained, with $r = 1$. We omit the proof.

\begin{theorem} \label{kiss}  Let $\|\cdot\|$ be any norm on 
 $\mathbb{R}^d$. Then there exists a discrete measure $\mu$ 
 such that $\|A^{cl}_{1,\mu}  \|_{L^1(\mu)\to L^1(\mu)} =  E(\mathbb{R}^d, \| \cdot\|)$.
\end{theorem}

The equality  $E(X, \| \cdot\|) = L(X, \| \cdot\|)$
allows one to transfer uniform bounds known for the centered maximal operator to uniform bounds for the averaging operators. 

In one dimension it is obvious that $E(X, \| \cdot\|) = 2$.
This observation extends to arbitrary measures on the real line  the upper bound 2 that appears in 
Theorem 4.2 for  the standard exponential
	distribution (given by $d P(t) = \mathbf{1}_{(0,\infty)} (t) \ e^{-t} dt$).

In higher dimensions, from Corollaries 3.4, 3.5 and 3.6 of \cite{Al1} we obtain the following

\begin{corollary} \label{infinitybounds} Given any norm $
	\|\cdot\|
	$ on the plane,  if the unit ball is a parallelogram then   $\sup_{r, \mu}\|A_{r,\mu}  \|_{L^1(\mu)\to L^1(\mu)} = 4
	$, while 
$\sup_{r, \mu}\|A_{r,\mu}  \|_{L^1(\mu)\to L^1(\mu)} =  5$ in every other case.

With balls defined using the $\ell_\infty$ norm,  the sharp uniform bound for $\sup_{r, \mu}\|A_{r,\mu}  \|_{L^1(\mu)\to L^1(\mu)}$ on $(\mathbb{R}^d , \|\cdot\|_\infty)$ 
	is $2^d$. Furthermore, the bound is attained.
	
	For the euclidean norm we have  
$\sup_{r, \mu}\|A_{r,\mu}  \|_{L^1(\mu)\to L^1(\mu)} =   12$ in dimension 3, and the bound is attained. Asymptotically, in dimension $d$ the following bounds hold:
\begin{equation}\label{asym}
	(1 + o(1)) \sqrt{\frac{3 \pi}{8}}  \log {\frac{3}{2 \sqrt 2}}
\ 	d^{3/2} \  \left(\frac{2}{\sqrt{3}}\right)^d
\le
\sup_{r, \mu}\|A_{r,\mu}  \|_{L^1(\mu)\to L^1(\mu)} \le
2^{0.401 (1 + o(1)) d}.
\end{equation}
\end{corollary}

\begin{remark} For $\mathbb{R}^d$ with the euclidean norm and
	the standard gaussian measure $\gamma$, it was shown in \cite[Theorem 4.3]{Al3}
that 
$\sup_{r>0}\|A_r\|_{L^1(\gamma)\to L^1(\gamma)} \le (2 + \varepsilon)^d$, whenever  $\varepsilon > 0$ and $d$ is large enough.  
The  upper bounds from the preceding result
(valid for all measures) represent a substantial improvement.
\end{remark}

\end{document}